\documentclass[11pt,letterpaper]{amsart}
\usepackage{amssymb,
                       amsmath,
                               txfonts,
                                      mathrsfs,
                                              titletoc,
                                                    appendix,
                                                    relsize}
\usepackage{graphicx}
\usepackage{xcolor}
\usepackage{subcaption}

\usepackage[alphabetic]{amsrefs}

 \newtheorem{thm}{Theorem}[section]
 \newtheorem{cor}[thm]{Corollary}
   
  \newtheorem{cjt*}{Conjecture}
 \newtheorem{lem}[thm]{Lemma}
 \newtheorem{prop}[thm]{Proposition}
 \newtheorem{defn}[thm]{Definition}
 \newtheorem{rem}[thm]{Remark}
 \numberwithin{equation}{section}
\newtheorem{lem*}{Lemma}
\newtheorem{cor*}{Corollary}

\newcommand{\parallelsum}{\mathbin{\!/\mkern-5mu/\!}}

\newenvironment{pfT}{\medskip \noindent
{\it Proof of Theorem \ref{T}.}}{\hfill $\square$\par
} 

\newenvironment{pfp42}{\medskip \noindent
{\it Sketch Proof of Proposition \ref{p42}.}}{\hfill $\square$\par
}

\newenvironment{pf}{\medskip \noindent
{\bf Proof of Theorem \ref{mainthm1}.}}{\hfill \rule{.5em}{1em}
\\}

\newenvironment{rmk}{\mbox{ }\\{\bf  Remark}\mbox{ }}{
\hfill $\Box$\mbox{}\bigskip}

\def\p#1{\partial #1}

\def\la{\lambda}

\def\th{\theta}

\def\R{\Bbb{R}}
\def\C{\Bbb{C}}

\begin{document}

\title[Obstruction by minimal product structure]{Cones over minimal products\\ cannot be calibrated by smooth calibrations}%
\author{Yongsheng Zhang}
\address{Academy for Multidisciplinary Studies, Capital Normal University, Beijing 100048, P. R. China}
\email{yongsheng.chang@gmail.com}
\date{\today}
 %

\begin{abstract}
We extend a key result in \cite{Z26},
by establishing the obstruction that
the minimal product structure (for minimal submanifolds or stationary currents in spheres) 
automatically makes all cones over (non-trivial) minimal products
fail to be calibrated by any global defined smooth calibration in Euclidean spaces.
\end{abstract}
\maketitle
\section{Introduction}\label{S1} 
    The class of area-minimizing cones naturally arising as tangent cones of area-minimizing currents (see \cite{FF})
    or tangent cones at infinity for certain situations
    is an important object which 
    has been successively studied for decades  
                              in several famous questions in differential geometry and geometric measure theory, e.g. see \cite{fle, de2, A, JS, BdGG, BL,  HL, LS,  HS, FK, Ch, Law0, Law, NS}.
                              
                              Until today, the most effective way to find area-minimizing cones is the theory of calibrations, e.g. special Lagrangian cones, (co)associative cones in  \cite{HL}.
                             Beyond calibrations with constant coefficients,
                             Lawlor's curvature criterion \cite{Law} essentially constructs generalized calibrations and  can provide us with  regular area-minimizing cones
                             (and non-orientable  area-minimizing cones 
                             mod 2).
                             
     Due to the dimension reason, 
     every area-minimizing hypercones cannot be calibrated by a globally defined smooth calibration.
     However for high codimensional situation the question was open and wondered by some experts 
     \begin{quote}
     $(\star)$ whether an area-minimizing cone of high codimension always supports some globally defined smooth calibration.
          \end{quote}
     Very recently in \cite {Z26} the author discovered  counterexamples which cannot support any globally defined smooth (or continuous) calibration.
  
     Such examples come from a combination of \cite{Z25} and \cite{Z26}.
     On the one hand, 
     using minimal product structure and Lawlor's criterion we gain  the following general configuration result.
     
          \begin{thm}[\cite{Z25}]\label{Z25main}
                                Given (oriented) embedded closed minimal submanifolds $\{L_i\}$ respectively in Euclidean spheres with $i=1,2,\cdots, n$,
                                every cone over the minimal product $L$ of sufficiently many copies among these $\{L_i\}$ is area-minimizing.
                                \end{thm}

                                 On the other hand, utilizing a classical result \cite{JS} by J. Simons we show 
                               that whenever some $L_i$ is a minimal hypersurface of positive dimension the cone over the minimal product
                               supports no globally defined  calibration.

     \begin{thm}[\cite{Z26}]\label{T2}
                                     Let $L_i$ be a closed embedded  minimal submanifold in $\mathbb S^{N_i}$ for $i=1,\cdots, n$.
                                     Suppose that $L_1$ is of codimension one. 
                                     Then the cone $C(L_1\dot\times\cdots\dot \times L_n)$ over their minimal product $L_1\dot\times\cdots\dot\times L_n$ supports no globally defined 
                                     smooth  calibration.
                                     \end{thm}
     
        Based on these two results,
        many counterexamples to question $(\star)$ can be found 
        and consequently in \cite{Z26} we extend the previous work  \cite{Z12} on detecting duality obstruction of calibrations in the smooth category.
        
        {\ }
        
        In this paper, we shall generalize Theorem \ref{T2} to its broadest version.
         
         \begin{thm}\label{T}
                                     Let $L_i$ be a closed embedded  minimal submanifold of positive dimension in $\mathbb S^{N_i}$ for $i=1,\cdots, n$ with $n\geq 2$.
                                     Then the cone $C(L_1\dot\times\cdots\dot\times L_n)$ over their minimal product  supports no globally defined smooth calibration.
                                     \end{thm}

        Therefore,
        the minimal product structure itself forms an obstruction
        to question $(\star)$
        and  in conjunction with Theorem \ref{Z25main} 
   more counterexamples can be obtained.

       {\ }
        
         The organization is the following.
                         We briefly recall some basic knowledge of calibrations
                          and the minimal product structure
                           in \S \ref{S2}.
                           A proof of Theorem \ref{T}
                          will be given in \S \ref{S3} 
                     and as a result the work of \cite{Z26}
                     can be further deepened based on Theorems \ref{Z25main} and \ref{addcor}.
                     As the Cartesian product of certain minimizing cones can be shown minimizing,
                     in \S \ref{S4} we extend Theorem \ref{T} involving the Cartesian product.
                     Finally, in \S \ref{S5}
                        we will extend Theorem \ref{T} to stationary current version.

              {\ }

                     \section{Preliminaries}\label{S2}

             
            In this section we first review some definitions and properties of calibrations,
            and then recall the minimal product structure for minimal submanifolds in Euclidean spheres.

            \subsection{Calibrations.}\label{cal}
            There is a norm called comass for a form
            which can be defined pointwise as follow. 
             \begin{defn}\label{comassdef}
            Let $\phi$ be a smooth $m$-form on a Riemannian manifold $(X,g)$.
            At a point $x\in X$ the comass of $\phi_x$ is defined to be
                     \begin{equation*}
                     \|\phi\|_{x,g}^*=\max \ \{\phi_x( \overrightarrow V_x) : \overrightarrow V_x \ \text{is a unit simple m-vector at x}\}.
                     \end{equation*}
             Here {``simple"} means $\overrightarrow V_x=e_1\wedge e_2\cdots \wedge e_m$
             for some $e_i\in T_xX$.
             \end{defn}
       
          \begin{defn}\label{calibration}
           A smooth form $\phi$ on $(X,g)$ is called a {calibration} if 
           $\sup_{X}\|\phi\|_{g}^*= 1$ 
           and
           $d\phi=0.$
           Such a triple $(X,\phi,g)$ is called a {calibrated manifold}.
          \end{defn}

                   The dual complex $(\mathscr E'_*(X),d)$ of smooth forms defines currents.
               Elements of $\mathscr E'_m(X)$ are $m$-dimensional {\it de Rham} 
                                                                                   {currents} 
                                                                                   (with compact support)
               and $d$ is the adjoint of exterior differentiation.
               There is a  dual norm called mass for currents
               given by
$\mathrm{\mathbf{M}}(T)=\sup \{T(\psi):\psi\ smooth\ m\text{-}form\ with\  \sup_{x\in X}\|\psi\|_{x, g}^*\leq 1\}.$

 When $T$ has finite mass, 
          there exist a unique {Radon} measure $\|T\|$ satisfying
                 $ \int_X f\cdot d\|T\|=\sup\{T(\psi): \|\psi\|_{x,g}^*\leq f(x) \}$
          for any nonnegative continuous function $f$ on $X$
          and
          a $\|T\|$ measurable tangent $m$-vector field $\overrightarrow T$ a.e. with 
          vectors $\overrightarrow T_x \in \Lambda^m T_xX$ of unit length in the dual norm of the comass norm, with
     $
                T(\psi)= \int_X\psi_x(\overrightarrow {T_x})\ d \|T\|(x)$
for any smooth $m$-form $\psi$
           \text{or\ briefly}
           $T = \overrightarrow T\cdot \|T\|\ a.e.\ \|T\|.$

Let $\mathbb M_m(X)=\{T\in\mathscr E'_m(X): \mathrm{\mathbf{M}}(T)<\infty\}$.
                        The milestone  paper \cite{FF}  defines 
                        normal currents
                         $N_m(X)=\{T\in\mathbb M_m(X): dT\in\mathbb M_{m-1}(X)\}$
     and the chain complex $(N_*(X),d)$ 
                 and moreover sets up
                 the isomorphisms
                            \begin{equation}\label{iso}
                            H_*(\mathscr E'_*(X))\cong H_*(X;\mathbb R)\cong H_*(N_*(X)).
                            \end{equation}
                            
                                \begin{defn}\label{calibratable}
                   Let $\phi$ be a calibration on $(X,g)$.
                   We say that a current $T$ of local finite mass is {calibrated} by $\phi$, if 
                        $\phi_x(\overrightarrow T_x)=1\ a.a.\ x\in X\ \text{for}\ \|T\|.$
                     In this case we say $T$ {supports} calibration $\phi$.
          \end{defn}

        The following Fundamental Theorem in Calibrated Geometry 
is powerful in proving mass-minimality.
 \begin{thm}[\cite{HL}]\label{hl}
                                           If $T$ is a calibrated current 
                                with compact support
                                           in a calibrated manifold $(X,\phi,g)$ and
                                           $T'$ is any compactly supported current homologous to $T$(i.e., $T-T'$ is a boundary and in particular $dT=dT'$),
                                           then
                                                          \begin{equation*}
                                                          \mathrm{\mathbf{M}}(T)\leq  \mathrm{\mathbf{M}}(T')
                                                          \end{equation*}
                                          with equality if and only if $T'$ is calibrated as well.
                               \end{thm}

                               Rectifiable $m$-currents  $\mathscr R_m(X)\subset \mathbb M_m(X)$
                           are
                             the closure of the integral Lipschitz $m$-chains in the $\bold M$-topology. 
                            Then $I_m(X):=\big\{T\in \mathscr R_m(X): dT\in \mathscr R_{m-1}(X)\big\}$
                               is the set of {integral} $m$-currents.
                               \cite{FF} also establishes 
                               \begin{equation}\label{isoz}
                               H_*(I_*(X))\cong H_*(X;\mathbb Z).
                               \end{equation}
                               
                                 Thus, a calibrated integral current
                                minimizes the mass not only in the homology class of integral currents but also that for normal currents.
                               If one focuses on its minimality among integral currents then it is called (homologically) area-minimizing.
                               
                               For more applications, certain singularity is allowed in practice.
                                 \begin{defn}\label{coflat}
           Let $\phi$ be a calibration of degree $m$ on $X-S_\phi$,
           where $S_\phi$ is a closed subset of $X$ of Hausdorff $m$-measure zero.
           Then $\phi$ is called a {coflat calibration} on $X$.
           We say $\phi$ {calibrates} a current, if it is calibrated by $\phi$ on $X-S_\phi$.
           Similarly, in such a case we say the current {supports} the coflat calibration $\phi$.
          \end{defn}
\begin{rem}
There is a corresponding coflat version of Theorem \ref{hl} with the same kind of conclusion.
\end{rem}
\begin{rem}
Via mollifications certain continuous and even non-continuous forms (cf. \cite{Law}) can serve as generalized calibrations.
\end{rem}

          In this paper we shall also make some results involving calibrated normal currents.
          So let us recall some related things here from II.7 \cite{HL}.
          Let $K$ denote the convex hull of $G(m, \R^{N+1})$ in $\wedge^m \R^{N+1}$
          where $G(m, \R^{N+1})$ stands for the oriented Grassmannian of simple $m$-vectors of unit size in $\R^{N+1}$.
          Then the unit ball under the mass norm is exactly $K$.
          
          Let $\phi\in\wedge^m (\R^{N+1})^*$ of comass one.
          Its dual facet is $\mathfrak F^*(\phi)=\{\xi\in \p K\,:\, \phi(\xi)=1\}$.
          Set $\mathscr G(\phi)=\{\xi\in G(m, \R^{N+1})\,:\, \phi(\xi)=\|\phi\|^*\}$.
          The the following tells us how $\overrightarrow T$ could be for a calibrated normal current $T$.
           
           \begin{lem}[Lemma 7.1 of \cite{HL}]\label{facet}
           Suppose $\phi\in\wedge^m (\R^{N+1})^*$ has comass one.
           Then
           $\mathfrak F^*(\phi)$ is exactly the convex hull of $\mathscr G(\phi)$.
           \end{lem}
           
              {\ }
              
              \subsection{Minimal products. }
                  Given minimal submanifolds $\{L_i\}$ in $\mathbb S^{N_i}$ respectively for $i=1,\cdots, n$,
                               a nice algorithm to generate new minimal submanifold  in sphere is 
                         \begin{equation}\label{prod}
         L_1\dot \times\cdots \dot\times L_n
                \triangleq
                 \left(\lambda_1 L_1, \cdots, \lambda_n  L_n\right)\subset \mathbb S^{N_1+\cdots +N_n+n-1}.
       \end{equation}
         where $\lambda_i=\sqrt{\frac{k_i}{k}}$ with $k_i$ the positive dimension of $L_i$.
              For various proofs of the minimal product structure,
              we refer to  \cite{X}, \cite{TZ} and \cite{CH}.

                              %
                            When a closed submanifold $L\subset \mathbb S^N$ is not totally geodesic,
                              the cone 
                              $$C(L)=\{tx:x\in L, \, t\in [0,\infty)\}$$ has only one singular point at the origin
                              and it is called a regular cone.
  Inspired by \cite{Law},
the author      studies the asymptotic behavior of regular minimal cone $C(L\times  \cdots \times L)$        for any embedded closed minimal submanifold $L$ in Euclidean sphere
and finally gains Theorem \ref{Z25main}.
Among others, we also get the following neat result.
       
          \begin{thm}[Theorem 1.3  of \cite{Z25}]\label{addcor}
                               Let $L$ be an (oriented) embedded closed  minimal subamanifold in $\mathbb S^n$.
                              Then the cone over minimal product $L \dot\times \mathbb S^{d}$ in $\R ^{n+d+2}$ is area-minimizing when $d$ is sufficiently large.
                               \end{thm}        
            Actually, cones in the statements of Theorems \ref{Z25main} and \ref{addcor} are mass-minimizing
            in the sense that each truncated cone within the unit ball is mass-minimizing among all currents sharing the same   boundary.
          A constellation of regular minimizing cones can be constructed
             based on any finite collection of minimal submanifolds in spheres.
           Here and in Theorem \ref{Z25main} there is no additional assumption at all on the embedded closed minimal submanifolds. 
               Any continuous deformation family of  closed minimal submanifolds in a sphere
               can lead to a deformation of corresponding minimizing cones in Euclidean space.

                          {\ }
                          
                                               
                                  \section{Proof of Theorem \ref{T}}\label{S3}         
                           In this section, we will prove Theorem \ref{T}.
                           There are two steps.
                           The first half is the same as that of Theorem \ref{T2} (i.e., Theorem 1.4 of \cite{Z26})
                           which based on a decomposition lemma for calibration forms.
                           
                           \subsection{Decomposition lemma}
                           A poweful canonical decomposition lemma for an $m$-form with respect to certain $m$-plane 
                           in $\R^n$
            is the following.

                             \begin{lem}[Lemma 2.12 in \cite{HL1}] \label{hl1}
             Let $\xi\in \Lambda^m \mathbb{R}^n$ be a simple $m$-vector with
             $V = span\ \xi$. Suppose $\phi\in\Lambda^m\big(\mathbb{R}^n\big)^*$ satisfies  $\phi(\xi)= 1$. 
              Then there exists a unique oriented complementary subspace $W$ to $V$  with the following property.
                      For any basis $v_1, \cdots, v_n$ of $\mathbb{R}^n$ such that $\xi=v_1\wedge... \wedge v_m$
                      and $v_{m+1}, \cdots, v_n$ is basis for $W$, 
                      one has that
                                 \begin{equation}
                                 \phi=v_1^*\wedge \cdots \wedge v_m^*+ \sum a_Iv_I^*,
                                  \end{equation}
                                   where $a_I=0$ whenever $i_{m-1}\leq m$. 
                                   Here $I=\{i_1,\cdots, i_m\}$ with $i_1<\cdots<i_m$.
              \end{lem}

                                     For a calibration form, more geometric information can be read off.
                 
     \begin{lem}[Decomposition for calibration form] \label{L2}
                       Let $v_1, \cdots, v_n$ be an orthonormal basis of $(\mathbb{R}^n, g_E)$ with $\xi=v_1\wedge... \wedge v_m$
                       and $V = span\ \xi$.
                       Assume that $\phi$ is a calibration with constant coefficients which calibrates the oriented $V$.
                       Then $W$ in Lemma \ref{hl1} has to be $\text{span}\{v_{m+1}\, \cdots,  v_n\}$ 
                       and as in Lemma \ref{hl1} we have
                         \begin{equation}\label{decomp}
                                 \phi=v_1^*\wedge \cdots \wedge v_m^*+ \sum a_Iv_I^*,
                                  \end{equation}
                                  where $a_I=0$ whenever $i_{m-1}\leq m$. 
                                   Here $I=\{i_1,\cdots, i_m\}$ with $i_1<\cdots<i_m$.
                                                         \end{lem}
                      
{\ }

           \subsection{Modified proof for Theorem \ref{T}. }
           Now we are ready to modify the the proof of Theorem \ref{T2} for Theorem \ref{T}.

        \begin{pfT}
        Suppose 
        $L_i$ is an embedded closed minimal submanifold of positive dimension $k_i$ in $\mathbb S^{N_i}$ for $i=1,\cdots, n$
        and 
        $C(L_1\dot\times\cdots\dot \times L_n)$ is calibrated by a globally defined smooth or continuous calibration $\phi$.
        Then the restriction $\phi_o$  at the origin $o$ calibrates $C(L_1\dot\times\cdots\dot\times L_n)$ in $\R^{N+1}$
        where $N+1=\sum_{i=1}^n (N_i+1)$.
                              Note that $\phi_o$
 is a calibration of constant coefficients with respect to $\{e_I\}$
 where 
 we put orthonormal bases of $\R^{N_i+1}$ together for an orthonormal basis $\{e_1, \cdots, e_{N+1}\}$ 
 of $\R^{N+1}$ and $e_I=e_{i_1}\wedge\cdots\wedge e_{i_m}$
 for $I=i_1\, \cdots \, i_m$ where $1\leq i_1<\cdots<i_m\leq N+1$ and $m=k_1+\cdots +k_n+1$.

Fix 
                               $x_i\in L_i$ 
                               and 
                               let $\xi_i$  be    the oriented unit tangent $k_i$-vector of $T_{x_{i}}L_{i}$    for $i=2, \cdots, n$.
                            For $x$ running in $L_1$,
                          we use
                                $\xi(x)$ to  stand for
                the oriented unit tangent $k_1$-vector of $T_{x}L_{1}$.
                                Let  $\mathcal P_1: \R^{N_1+1}\hookrightarrow \R^{N+1}$ 
                            be the inclusion given by $\mathcal P_1(x)=(x, 0, \cdots, 0)$
                            and similarly define $\mathcal P_i:\R^{N_i+1}\hookrightarrow \R^{N+1}$ for $i=2,\cdots, n$.
                            Denote 
                            $\left(\mathcal P_1\right)_\#\xi(x)$
                            and
                             $\left(\mathcal P_i\right)_{\#} \xi_i$
                             by 
                             $\tilde \xi(x)$
                             and
                                $
                          \tilde \xi_i$
                          for $i\geq 2$ respectively.
                              In \cite{Z26} 
                              we consider 
                             calibration $(k_1+1)$-form 
                             $\Phi_o=\left(\tilde \xi_2\wedge \cdots \wedge \tilde \xi_n\right) \lrcorner \phi_o$ 
                             of constant coefficients.
Note that  $\Phi_o$ or $-\Phi_o$ 
               calibrates $p(x)\wedge \tilde \xi(x)$ at 
               $p(x)=(\la_1x,\la_2x_2,\cdots, \la_nx_n)\in L_1\dot\times \cdots\dot\times L_n$
               for each $x\in L_1$.

Also note that, for each $x\in L_1 $,
                           any orthonormal basis $ \{v_1(x),\cdots, v_{k_1}(x) \}$ of the linear space span $\tilde \xi(x)$
                           can expand  to an orthonormal basis 
                          \begin{equation}\label{basis}
                          \Big\{v_1(x),\cdots, v_{k_1}(x),  \, v_{k_1+1}(x),\, \ v_{k_1+2}(x), \,\cdots, v_N(x)\Big\}
                           \end{equation}
                           of $\R^{N+1}$
                          with $v_{k_1+1}(x)=p(x)$ and 
                           $v_{k_1+2}(x)=p^\perp(x)$  a unit orthogonal vector to $p(x)$ in the plane span$\{p(x), \tilde x= (\mathcal P_1)_\# x\}$.
                           According to  Lemma \ref{L2}
                we have 
                            \begin{equation}\label{deC}
                                 \Phi_o\big|_{p(x)}= \pm v_1^*(x)\wedge \cdots \wedge v_{k_1+1}^*(x)+ \sum a_Iv_I^*(x),
                                  \end{equation}
                                  where $a_I=0$ whenever $i_{k_1}\leq k_1+1$. 
                                  For nonzero $a_I$, the corresponding multi-index
                                  $I$ must have at least two indices  strictly larger than $k_1+1$.
                               Note that among elements of \eqref{basis} only $v_{k_1+1}(x)$ and $v_{k_1+2}(x)$ have nonzero evaluations with  $\tilde x$
                               but the coefficient $a_I$ vanishes for $I=1\, 2\, \cdots\, k_1\, \widehat{\, (k_1+1)\, }\, (k_1+2)$.
                               Therefore, we get  $\left(\mathcal P_1^*(\Phi_o)\right)(\xi(x)\wedge x)=\Phi_o(\tilde \xi(x)\wedge \tilde x)$ being constant in variable $x$.
                
               If we instead focus on 
                $$\Psi=\left(\tilde x_2\wedge \tilde \xi_2\wedge \cdots \wedge \tilde \xi_n\right) \lrcorner \phi_o$$
                which is a $k_1$-form of constant coefficients in $\R^{N+1}$ as now points $\{x_i\in L_i\}_{i=2}^n$ are fixed,
                then $\underline{\Psi}=\mathcal P_1^* \Psi$ is a $k_1$-form of constant coefficients in $\R^{N_1+1}$ 
                which essentially eliminates  all terms not completely in $\R^{N_1+1}$.
                It can be observed based on \eqref{deC} that
                $$
              \underline{\Psi}(\xi(x))
              =
              {\Psi}\big(\tilde \xi(x)\big)
              =
              \pm
              \left(\tilde \xi(x)\wedge \tilde x_2\wedge \tilde \xi_2\wedge \cdots \wedge \tilde \xi_n\right) \lrcorner \phi_o
              =
              \pm \la_2
              \neq 0.
                $$
                Since 
                $
              \underline{\Psi}
              $
              is a closed form
              on the entire   $\R^{N+1}$
              and 
              $\R^{N+1}$
              has trivial homologies,
             it finally leads us      
             to the conclusion 
             $0=\int_{L_1}\underline{\Psi}=\pm \la_2 \text{vol}(L_1)\neq 0$.
             Contradiction!
             Hence we finish the proof.        
        \end{pfT}

\begin{rem}
$L_1$ of positive dimension is null-homologous in $\R^{N_1+1}$ and  it is in particular the boundary of truncated $C(L_1)$ within the unit ball.
\end{rem}
\begin{rem}\label{ic}
The same result holds for integral current inputs $L_1,\cdots, L_n$ 
as one can take $x_i$ to be a regular point of $L_i$ for $i=2,\cdots, n$ in the proof 
and the null-homotopy of $L_1$ still leads us to the contradiction in the last step.
Although our arguments can work well for any integral currents and positive combination of factors in \eqref{prod},
the useful case is essentially  the minimal product taking  stationary integral currents as inputs.
See \S \ref{S5} for more details.
\end{rem}

                         \begin{rem}
                          In particular Theorem \ref{T} completely answers the question of Remark 5.2 in \cite{Z26}.
                         As a result all the minimizing cones in Theorems \ref{Z25main} and \ref{addcor} (not limited to Theorem \ref{T2})
                         can be used to form the obstruction of calibrations in the smooth category in \cite{Z26}.
                         \end{rem}
                 {\ }

                           \section{Situation involving Cartesian product} \label{S4}
                           In this section,
                           we consider the situation involving  Cartesian product.
                           Given two  mass-minimizing currents in Euclidean spaces respectively,
                            it is a long-standing question
                         whether their Cartesian product is mass-minimizing in the product Euclidean space.
                        In general it is still largely open
                        and
                         Theorem 5.2 of \cite{M}
                         answers this question in some special situations:
                           if either of them has dimension or codimension no larger than two, or both of them have dimension simultaneously or codimension simultaneously  equal to three,
                           then their Cartesian product is also mass-minimizing.
                           
                        \cite{Law} provides a sufficient criterion for a regular minimal cone to be mass-minimizing.
                     Those orientable  cones  which satisfy Lawlor's criterion are called Lawlor cones.   
                    Examples of minimizing cones in Theorems \ref{Z25main} and \ref{addcor}  are all  Lawlor cones.
                      Although those minimizing cones in Theorems \ref{Z25main} and \ref{addcor} 
      cannot be calibrated by globally defined continuous calibrations,
      we show in \cite{Z26} that each of them supports a coflat calibration which is smooth and simple away from the origin.
This means the singularity of dual calibrations has to be nonempty.
      Moreover,
                                a nice property of  Lawlor  cones
                                is that the Cartesian product among them is again mass-minimizing.

                              \begin{prop}[Corollary 6.6 of \cite{Z26}]\label{minCartProd}
                                  Let  $C_i\subset \R^{N_i+1}$  be an (oriented) Lawlor cone for $i=1,\cdots, n$.
                                  Then 
                                 the Cartesian 
                                 cone $C_1\times\cdots\times C_n\subset \R^{N_1+1}\oplus\cdots\oplus \R^{N_n+1}$ (not a regular cone) is  mass-minimizing. 
                                  \end{prop}
                                  Actually, we can extend Proposition \ref{minCartProd} a bit.

                                  \begin{prop}\label{p42}
                                  Let   $C_0=C_0(L_0)\subset \R^{N_0+1}$ be an arbitrary mass-minimizing cone (over a normal current $L_0$ in $\mathbb S^{N_0}$)
                                  and
                                  $C_i\subset \R^{N_i+1}$  be an (oriented) Lawlor cone for $i=1,\cdots, n$.
                                  Then 
                                 the Cartesian cone $C_0\times C_1\times\cdots\times C_n\subset \R^{N_1+1}\oplus\cdots\oplus \R^{N_n+1}$ (not a regular cone) is  mass-minimizing. 
                                     \end{prop}
                                                \begin{pfp42}
                                                Following the thoughts in \cite{M},
                                                for the mass-minimizing cone $C_0$ there exists a closed flat cochain represented by a measurable  calibration form $\phi_0$.
                                                Suppose the simple coflat calibrations for $\{C_i\}$ are denoted by $\phi_i$ for $i=1,\cdots, n$.
                                                Then 
                                                by the simpleness of $\phi_1,\cdots, \phi_n$
                                                the exterior product
                                                $\phi_0\wedge \phi_1\wedge \cdots \wedge \phi_n$ is a measurable calibration form for $C_0\times C_1\times \cdots \times C_n$
                                                and
                                                 one can get a proof  using mollifications within the Euclidean space.
                                                \end{pfp42}
                              \begin{rem}
                              The mass-minimality statement is still true if we replace $C_0$ by any mass-minimizing normal current $T$ in $\R^{N_0+1}$.
                             \end{rem}
                                    
                                    The proof of Theorem \ref{T} implies the following corollary.
                                    
                                    \begin{cor}\label{c0prod}
                                    Assume $C_0=C_0(L_0)$ and $C_i$ are given as in Proposition \ref{p42} with $n\geq 1$.
                                    Suppose the link  of $C_1$ is a minimal product of  positive dimensional minimal submanifolds.
                                    Then the mass-minimizing Cartesian cone $C_0\times C_1\times \cdots \times C_n$ cannot be calibrated by any global defined continuous calibration.
                                    \end{cor}
                                          \begin{rem}\label{complicated}
                                        The support of a general normal current can be very complicated.
                                        An example that the author learned from Prof. Blaine Lawson is 
                                       the normal current limit(s) \cite{FF} of $T_k/\mathrm{\mathbf{M}}(T_k)$ in $\C P^2$ 
                                        where $T_k=\{[z_0,\, z_1,\, z_2]\,  :\, z_0^k+z_1^k+z_2^k=0\}$.
                                        Similarly, consider
                                        minimizing cones
                                        $\tilde T_k=\{(z_0,\, z_1,\, z_2)\in \C^3\,  :\, z_0^k+z_1^k+z_2^k=0\}$
                                        and denote its corresponding link in $\mathbb S^5$ by $\tilde L_k$.
                                        Then by the celebrated compactness theorem of \cite{FF} there exists normal current limit(s) of $\tilde L_k/\mathrm{\mathbf{M}}(\tilde L_k)$
                                        in $\mathbb S^5$.
                                        \end{rem}
                                    \begin{proof}
                                    Suppose the Cartesian product is calibrated by $\phi$.
                                    Then it is calibrated by the calibration $\phi_o$ of constant coefficients.
                                    By assumption we can take regular points $x_i\in \text{supp }C_i$ for $i=1,\cdots, n$
                                    and their unit tangential multi-vectors $\check\xi_i$ of $C_i$ at $x_i$.
                                    Then 
                                    the pullback $\underline\phi$ of 
                                    $\left(\check\xi_1\wedge \cdots \wedge \check \xi_n\right) \lrcorner \phi_o$
                                    to
                                    $\R^{N_0+1}$
                                    calibrates $C_0$ (with right orientation).
                                    By Lemma \ref{facet},
                                    at $\|C_0\|$-a.e. $x_0\in \text{supp }C_0$ where $\overrightarrow C_0$ exists,
                                    we have $\overrightarrow C_0\in \mathfrak F^*(\underline\phi)=$ covex hull of $\mathscr G(\underline\phi)$.
                                    Hence take $\xi_0\in \mathscr G(\underline\phi)$
                                    and use $\check \xi_0$ for its natural inclusion (push-forward) in the Cartesian product Euclidean space. 
                                    Now consider the calibration of constant coefficients
                                   $\left(\check\xi_0\wedge \check\xi_2\cdots \wedge \check \xi_n\right) \lrcorner \phi_o$ 
                                   and it calibrates $C_1$ (with right orientation).
                                   This is impossible according to the assumption on the link of $C_1$ and Theorem \ref{T} as well as Remark \ref{ic}.
                               Thus the proof gets accomplished.
                                                                        \end{proof}
                                        
                                  We would like to remark two things.
                                  One is,                                  unlike the cones over products of links in Theorem \ref{T},
                                  the Cartesian product itself does not eliminate the possibility 
                                  of being calibrated by globally defined smooth calibrations.
                                  Besides trivial examples, 
                                  let us consider the Cartesian product of special Lagrangian cones which is again special Lagrangian.
                      The other is, even within the category of regular cones,
                     regular mass-minimizing cones proven by Lawlor's criterion and calibrations of constant coefficients
                     respectively
                     are different sets.
                     On the one hand, as pointed out in \cite{XYZ2},
                     the mass-minimality of the first Lawson-Osserman cone in \cite{LO}
                     (calibrated by the coassociative calibration in $\R^7$, see \cite{HL}) 
                     cannot be gained through Lawlor's criterion.
                     Many more examples of regular special Lagrangian cones
                     which  cannot be verified by the criterion
                     can be constructed by the spiral minimal product structure in \cite{LZ}.
                      Note that 
                      the link of a special Lagrangian cone is a special Legendrian current in the odd-dimensional spheres of complex Euclidean space.
                      There are  two parameters $C_1, C_2$ in the spiral minimal product algorithm.
                      When $C_1=-1$ it can generate new examples of embedded special Legendrian submanifolds based on given ones. 
                      By adjusting the value of $C_2$ one can make the generated special Legendrian submanifold have
                       normal radius (a key  quantity in the criterion of \cite{Law})
                       as small as wishes
                       and hence impossible to apply the criterion.
                       On the other hand, 
                       besides many minimizing hypercones,
                       our results here add more examples of minimizing cones which can be found through the criterion of \cite{Law} or Corollary \ref{c0prod}
                       but cannot by calibrations of constant coefficients.
                      
                           {\ }

                           \section{Stationary integral current version of Theorem \ref{T}}\label{S5}
                          We merely allude to stationary currents  in this paper occasionally
                           up to now,
                           and have not yet explained 
                           the definition and why the minimal product structure can produce a new stationary current based on two given ones in spheres rigorously.
                           So, in the section, we will make some explanations and provide some illustrations helpful for understanding.
                           Recall   that the current links that we consider in this paper are all $d$-closed with no boundary and compactly supported, see \S \ref{cal}.

                         By stationary integral $m$-current we mean integral current which is 
                         stationary in the variational sense by deformation of $C^1$-vector field of ambient space. 
                          Another name of this object is stationary integral varifold or integer multiplicity stationary varifold.
                            Following the definition in \cite{Sim},
                            each such object 
                            associates with
                            a rectifiable set $M$ of local finite Hausdorff measure 
                            and a locally bounded density function $\th: M\rightarrow \mathbb Z_+$ $\mathscr H^m$-a.e.
                            and
                            satisfies
                          \begin{equation}\label{stationary}
                            \frac{d}{dt}\Big|_{t=0} \int_{\Psi_t(M)} \theta \left(\Psi^{-1}_t(x)\right)
                             d\mathscr H^m x
                             =0
                       \end{equation}
                            for all deformation $\Psi_t(x)$  with $\Psi_t'(x)=X$ 
                            where $X$ is any compactly supported $C^1$-vector field.
                            An equivalent version of \eqref{stationary}
                            is
                                  \begin{equation}\label{stationary2}
                          \int_{M} div_M X 
                          \, 
                          d\mu
                             =0
                       \end{equation}
                       where $\mu=\th d\mathscr H^m x$
                       and $div_M X=\sum_{i=1}^m \tau_i\cdot D_{\tau_i} X$ where $\{\tau_1, \cdots, \tau_m\}$ is an orthonormal basis of $T_xM$ 
                       (which exist for $\mathscr H^m$-a.e. in $M$).
                       Note that \eqref{stationary2} is essentially integrated over $M\sim N_0$
                        where 
                        $M\subset \bigcup_{j=0}^{\infty} N_j$ with $\mathscr H^m(N_0)=0$
                        and each $N_j$ for $j>0$
                        an $n$-dimensional embedded $C^1$ submanifold
                        (see e.g. Chapter 3 in \cite{Sim16}).
                        Roughly speaking, 
                        \eqref{stationary2} splits into two parts for normal and tangential components $X^N$ and $X^T$
                        $-$ an interior term on $M\sim N_0$ and   a ``boundary term" along $\bigcup_{j=1}^\infty \p \big(M\bigcap N_j\big)$:
                        being stationary implies that both terms vanish for all $X$
                        and 
                        moreover in such case       
                        one can get each $N_j$ where $j>0$ is a smooth submanifold with vanishing mean curvature field
                        and
                        along 
                       the ``boundary term" 
                        $M\bigcap N_0$ when $\p M=\emptyset$ (or $(M\bigcap N_0)\sim \p M$ in the sense of fixed boundary)
                       distributions given by exterior unit normals are  balanced with $\th$-weight to be stationary $\mathscr H^{m-1}$-a.e.
                       as illustrated in the following
                       (otherwise one would have extra boundary in $\p M$).

     \begin{figure}[h]
		\includegraphics[scale=.8]{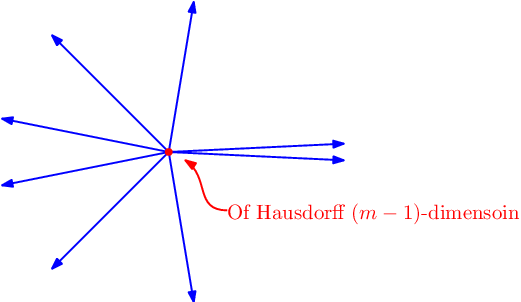}
                             \caption{A local balanced status at a point  with weight $\th\equiv 1$} 
                             \end{figure}

                       \begin{prop}\label{prodcurrent}
                       Given stationary $m_i$-dimensional integral varifolds $T_i$ in $\mathbb S^{N_i}$ for $i=1,2$.
                       Then the minimal product $T=\left(\la_1 T_1\right)\times \left(\la_2T_2\right)$
                       with $\la_i=\sqrt{\frac{m_i}{m_1+m_2}}$ 
                       is a stationary $(m_1+m_2)$-dimensional integral varifold in $\mathbb S^{N_1+N_2+1}$ .
                       \end{prop}
                       
                       \begin{proof}
                       First of all, the product makes sense \cite{F}.
                  %
                       Note that orthonormal bases
                       $\{\tau_1, \cdots, \tau_{m_1}\}$
                       and 
                       $\{\sigma_1, \cdots, \sigma_{m_2}\}$
                       of $x_i\in T_i$ respectively
                       form an orthonormal basis of $T$ at $(\la_1 x_1, \la_2 x_2)$
                       $\mathscr H^{m_1+m_2}$-a.e. in $T$.
                       Let $X$ be a $C^1$-vector field in $\mathbb S^{N_1+N_2+1}$.
                       Then 
                       away from $\mathcal P_1(\mathbb S^{N_1})$
                       and $\mathcal P_2(\mathbb S^{N_2})$
                       the tangential vector field 
                       $X$ decomposed into three parts according to $(y_1, y_2)\in \mathbb S^{N_1+N_2+1}$,
                       i.e., $X=X_1+X_2+\eta$
                       where 
                       $$X_i\in (\mathcal P_i)_\#T_{\frac{y_i}{\|y_i\|}}\mathbb S^{N_i}
                       \text{\ \ \ \ 
                       \
                       and
                       \
                        \ \ \ \ }
                       \eta \parallelsum \eta_0
                       $$
                       with $\eta_0=(-\|y_2\| \frac{y_1}{\|y_1\|}, \|y_1\| \frac{y_2}{\|y_2\|})$.
                       Correspondingly, 
                       in the $\la_1M_1\times \la_2M_2$
                       we have
                       $$\eta=f(\la_1x_1, \la_2x_2)
                     \cdot  \eta_0
                       =
                       f
                       \cdot (-\la_2 x_1, \la_1 x_2), 
                       \text{\ \ \ \ where   }
                       f 
                       \text{\ is a }
                       C^1
                       \text{\  real function}.
                       $$
                       To be simple
                       we use
                    $dy_1$ for $d \mathscr H^{m_1}y_1$ 
                    and similarly for other terms accordingly.
                       As $\overrightarrow T_i$ is simple $\|T_i\|$-a.e. for integral current $T_i$ with $i= 1$ and $2$,
                    by 4.1.8 of \cite{F}
                       we have a.e. that
\begin{equation}\label{Cprod}
                       \overrightarrow {T_1\times T_2}
                       =
                         \overrightarrow T_1
                         \wedge
                           \overrightarrow T_2
                         \
                         \text{
                           \ \ \ and
                           \ \ \
                           }
                           \|T_1\times T_2\|=\|T_1\|\times \|T_2\|.
     \end{equation}
                       Let $M_i$ be the rectifiable set associated to $T_i$ with weight $\th_i$ for $i=1$ and $2$.
                       We get
                    \begin{eqnarray} 
&&  \int_{\la_1M_1\times \la_2M_2} div_{\la_1M_1\times \la_2M_2} X \, \, \, \th_1  dy_1\,  \th_2 dy_2
    \label{var}
\\
 &=&
            \int_{\la_1M_1\times \la_2M_2} div_{\la_1M_1\times \la_2M_2} \left(X_1+X_2+\eta\right)
             \, \, \, \th_1  dy_1\,  \th_2 dy_2
       \nonumber  
       \\
        &=&
            \int_{\la_2M_2} \int_{M_1} \frac{1}{\la_1}div_{M_1} X_1 \, \, \, \th_1  \frac{dx_1}{\la_1^{m_1}}\,  \th_2 dy_2
            +
              \int_{\la_1M_1} \int_{M_2} \frac{1}{\la_2}div_{M_2} X_2  \, \, \, \th_2  \frac{dx_2}{\la_2^{m_2}}\,  \th_1 dy_1
                     \nonumber  
       \\
        && \ \ \ \ \ \ \ \ \ \ \ \qquad \ \ \ \ \ \ \ \ \ \ \ \ \ \ \ \ \ \ \ \ \ \ \  \ \ \ \ \ \
            +\int_{M_1\times M_2} \left(-\frac{\la_2}{\la_1} m_1+\frac{\la_1}{\la_2} m_2\right) f 
            \, \, \, \th_1  \frac{dx_1}{\la_1^{m_1}}\,  \th_2 \frac{dx_2}{\la_2^{m_2}}
                  \nonumber  
                    \\
        &=&
        0\,\, .
         \nonumber  
\end{eqnarray} 
Here we use $div_{\la_1M_1}X_2
                              =
                                           \sum_{i=1}^{m_1} \tau_i\cdot D_{\tau_i} X_2
                                           =-
                                           \sum_{i=1}^{m_1}  D_{\tau_i} \tau_i \cdot X_2=0$
                                           and
similarly for div$_{\la_2M_2}X_1$.
Thus we finish the proof.
                       \end{proof}
                       
                       Examples of stationary currents in spheres are many.
                       For instance,
                       let $$\{e_1, \, f_1,    \, e_2, \, f_2,        \, e_3, \, f_3\}
                       \text{ be the basis of }\C^3=\C\oplus\C\oplus\C
                       $$  
                       with each $\C=$span$\{e_j,\, f_j\}$ and $ie_j=f_j$ for $j=1,2,3$.
                       Then the $3$-vectors
                       $$
                       e_1\wedge e_2\wedge e_3,
                       \, \ \ \
                       -f_1\wedge f_2\wedge e_3,
                         \, \ \ \
                    -e_1\wedge f_2\wedge f_3,
                       $$
                       span three special Lagrangian $3$-planes of special Lagrangian angle/phase zero.
                       Hence their union gives a special Lagranian current 
                       with link $L$ being union of three $2$-spheres intersecting in pairs exactly in three points in $\mathbb S^5$. 
                       Given any special Legendrian integral current  in an odd-dimensional sphere,
                       one can follow \cite{LZ} to have spiral minimal product current
                       of it with the $L$ here.
                       \footnote{
                       A rigorous proof of this can be done similarly as that in the proof of Proposition \ref{prodcurrent}.
                       }
                       In this way, many stationary currents in unit spheres with mild/trackable singularity can be constructed.
          
          {\ }
          
          Now we  remark that Proposition \ref{prodcurrent} also holds for stationary normal currents $T_1$ and $T_2$.
          First, note that the Cartesian product is well defined for {\it de Rham} currents.
          Although \eqref{Cprod} is generally unknown for normal currents,
          the meaning of being stationary is formally the same.
      Let $T$ be an $m$-dimensional normal current in ambient Riemannian manifold $(Z,g)$. 
      Then by \S \ref{cal} the normal current $T=\overrightarrow T\cdot \|T\|$  a.e. for $\|T\|$ or $\mathscr H^m$
      where
      $\overrightarrow T(p)=\sum c_j(p) P_j(p)\in \Lambda^m T_p Z$ of the unit mass under metric $g_p$.
      Note that, as $\Lambda^m T_p Z$ has finite dimension, 
      each 
      $\overrightarrow T(p)$
      can be written as a sum of finite terms.
       According to 4.1.7 of \cite{F} the condition for $T$ to be stationary is similar to \eqref{stationary2}:
          $$
          \int_{\text{supp} T} \sum c_j(p) \cdot div_{P_j(p)} X \, \, d\|T\| \, p=0
          $$
          for every compactly supported $C^1$-vector field $X$ on $Z$.
          Regardless of \eqref{Cprod},
           by the definition of Cartesian product
                     (cf. 4.1.8 of \cite{F})
           we always have
           $T_1\wedge T_2
           = 
            \overrightarrow T_1
                         \wedge
                           \overrightarrow T_2
                           \cdot
                           \|T_1\|\times \|T_2\|
                           $.
                           Hence,
                           the argument for \eqref{var} implies that the Cartesian product of  normal currents $\la_1T_1$ and $\la_2T_2$ is stationary in the sphere.
          
          {\ }
          
          Theorem \ref{T} is the broadest version in format.
         Focusing on what kind of inputs the theorem can take
         we derive the following.

                       \begin{cor}\label{corT}
                       Let  $T_0$  be a (stationary) normal current of positive dimension in $\mathbb S^{N_0}$
                                and $T_i$  a (stationary)  integral current of positive dimension 
                                    in $\mathbb S^{N_i}$ for $i=1,\cdots, n$.
                                     Then the (stationary) cone $C(T)$ where $T=T_0\dot\times T_1\dot\times\cdots\dot\times T_n$ 
                                     supports no globally defined continuous calibration.
                                     \end{cor}

                 \begin{proof}
                 Suppose it can be calibrated by globally defined continuous calibration $\phi$.
                 Then the cone is calibrated by the calibration $\phi_o$ of constant coefficients.

                 Choose $x_i$ to be a regular point of $T_i$ for $i=1,\cdots, n$.
                 Then denote the corresponding unit tangential multi-vectors by $\xi_i$
                 and the push-forward counterparts by $\tilde \xi_i$ respectively. 
                 Set $p(x)=(\la_0x, \la_1x_1,\cdots, \la_nx_n)$ for $x$ running in the support of $T_0$.
                  It is clear by Definition \ref{calibratable}
                  that the normal current $C(T)$ is calibrated by $\phi_o$
                implies that
                  $\overrightarrow T$ is calibrated by $p(x) \lrcorner\phi_o$ a.e. under $\|T_0\|$.
                      The reason is 
                                 $\overrightarrow T=\overrightarrow T_0 \wedge \tilde \xi_1 \wedge \cdots \wedge \tilde \xi_n$
                 (cf. 1.8.4 of \cite{F})
                 as every $\tilde \xi_i$ is simple for $1\leq i\leq n$.
                 Although $\overrightarrow T_0$ may not be a simple vector,
                 one can $\|T_0\|$-a.e.  have  
                it belong to the convex hull of
                $\mathscr G=\mathscr G(\varphi)$ (see \S \ref{cal} and  reverse the orientation of $C(T)$  if necessary)
                where
                $$
               \varphi=\left(\mathcal P_0^* \Big(\big(p(x)\wedge\tilde \xi_1\wedge\cdots \wedge \tilde \xi_n\big) \lrcorner \phi_o \Big)\right).
                $$
              Here
                  $\mathcal P_0$ can be defined similarly as in the proof of Theorem \ref{T}.
               By Lemma \ref{facet}, we have the decomposition
      \begin{equation}\label{decT0}
                \overrightarrow T_0=\sum_{j=1}^\ell c_j \zeta_j
           \end{equation}
                 where $\ell$ is a finite integer, each $c_j>0$ with $\sum_{j=1}^\ell c_j =1$ and $\zeta_j\in \mathscr G$.
                
                As $\varphi(\zeta_j)=1$,
                it can be seen that
                $ \Big(\big(\tilde \xi_1\wedge\cdots \wedge \tilde \xi_n\big) \lrcorner \phi_o  \Big)(p(x)\wedge \tilde \zeta_j)$ equals $+1$ or $-1$ uniformly for all $j$
                where $\tilde \zeta_j=\big(\mathcal P_0\big)_\# \zeta_j$.
          Now for each simple multi-vector $p(x)\wedge \tilde\zeta_j$
          we can apply Lemma \ref{L2} and repeat the argument in the proof of Theorem \ref{T}
          to show that, for $\tilde \Psi=\left(\tilde x_1\wedge \tilde \xi_1\wedge \cdots \wedge \tilde \xi_n\right) \lrcorner \phi_o$,
          we have
          $$
         \big( \mathcal P_0^*\tilde \Psi \big)(\zeta_j)= +\la_1 \text{ or  } -\la_1 
          $$
          uniformly for all $\zeta_j$.
          As a result, it follows that
           $
               \big( \mathcal P_0^*\tilde \Psi \big)(\overrightarrow T_0)
               =\pm\la_1$
           according to \eqref{decT0}.
           Therefore,
           similarly
           we  end up with the contradiction that 
                                    $$0=T_0\big(\mathcal P_0^*\tilde \Psi\big) = \pm\la_1  \mathrm{\mathbf{M}}(T_0)\neq 0$$
                                    and 
                                    thus finish the proof.
                 \end{proof}

            \begin{rem}
            If \eqref{Cprod} holds (i.e., equivalently the famous Federer's product question on mass in 1.8.4 of \cite{F} can be answered affirmatively),
            then
           we can further extend Corollary \ref{corT}
            by allowing these $T_i$ taken from (stationary) normal currents. 
            \end{rem}
           
                    {\ }

      {\ }

{\ }

\end{document}